\UseRawInputEncoding
\documentclass[reqno,a4paper,12pt]{amsart}
\usepackage[top=32mm, bottom=32mm, left=32mm, right=32mm]{geometry}
\usepackage{mathrsfs,enumerate}
\usepackage{amssymb}
\usepackage[colorlinks=true]{hyperref}

\allowdisplaybreaks

\newcommand{\ld}{\lambda}
\newcommand{\ga}{\gamma}
\newcommand{\al}{\alpha}
\newcommand{\be}{\beta}
\newcommand{\de}{\delta}
\newcommand{\si}{\sigma}
\newcommand{\ep}{\epsilon}

\newcommand{\Ga}{\Gamma}

\newcommand{\tht}{\theta}

\newcommand{\beqq}{\begin{equation*}}
\newcommand{\eeqq}{\end{equation*}}
\newcommand{\beq}{\begin{equation}}
\newcommand{\eeq}{\end{equation}}

\newtheorem{theorem}{Theorem}
\newtheorem{lemma}[theorem]{Lemma}

\theoremstyle{definition}
\newtheorem{definition}[theorem]{Definition}
\newtheorem{proposition}[theorem]{Proposition}

\newtheorem{corollary}[theorem]{Corollary}

\theoremstyle{remark}
\newtheorem{remark}[theorem]{Remark}

\begin{document}
\title[]{Ergodic maximization problem for expanding maps with differentiable observables}

\author[X. Zhang]{Xu Zhang}
\address[X. Zhang]{Department of Mathematics, Shandong University, Weihai, Shandong, 264209,  China}
\email{xu\_zhang\_sdu@mail.sdu.edu.cn}

\keywords{Entropy; ergodic maximization; invariant measure; observable; periodic orbit; shadowing}

\begin{abstract}
We show that for an expanding map, the maximizing measures of a generic (open and dense) $C^r$ ($r\in\mathbb{N}$) differentiable functions are supported on a single periodic orbit. [There is a gap in the discussions. For the $C^{\infty}$ approximation of the Lipschitz functions, we can only control the $C^1$ derivative, but we can not control the $C^r$ derivatives for $r\geq2$. Elegant approximation methods might be needed to solve this problem.]
\end{abstract}

\maketitle

\section{Introduction}

Ergodic theory relates closely to the iteration of a measure preserving transformation $T$ on a metric space $(X,d)$ with the probability structure $(X,\mathcal{B},\mu)$. If $\mu(T^{-1}(B))=\mu(B)$ for any $B\in\mathcal{B}$, then $\mu$ is called an invariant measure.  An observable is a continuous function $f:X\to\mathbb{R}$. The time average of the real-valued function $f$ along the orbit is $\lim_{n\to+\infty}\frac{1}{n}\sum^{n-1}_{i=0}f(T^i(x))$ if the limit exists. The space average of the real-valued function $f$ with respect to an invariant measure $\mu$ is $\int_Xfd\mu$. For an observable $f$, the maximizing orbit for $f$ is the orbit giving the maximum time average of $f$, and the maximizing invariant probability measure for $f$ is the measure giving the maximum space average of $f$. If the invariant measure is ergodic, the classical Birkhoff ergodic theorem tells us that the time average of the observable is equal to the space average of the observable for almost ever point in the view point of the ergodic measure. Ergodic optimization is the study on the problems of maximizing orbits and invariant measures, and can be applied in the control of chaos  \cite{ShinbortGrebogiOttYorke1993,OttGrebogiYorke1990}, the Aubry-Mather theory in Lagrangian mechanics \cite{Contreras2015, Mane1996}, and  the ground state theory in thermodynamics formalism and multifractal analysis \cite{BaravieraLeplaideurLopes2013}, and so on.

A core  problem in the field of ergodic
optimization theory is the Typical Periodic
Optimization (TPO) Conjecture for the uniformly expanding maps and the uniformly hyperbolic system,  which was proposed by Yuan and Hunt \cite{YuanHunt1999}. For the real-valued observables taken from some kind of smooth (H\"{o}lder continuous, Lipschitz continuous, differentiable) function spaces, denoted by $V$ the set of observables such that the maximizing measures for each element of $V$ contains a periodic measure,  the TPO Conjecture claims that $V$ is an open and dense subset of the function space.

There are many results towards TPO Conjecture.  A series of work on the subshift of finite type are obtained: Coelho's work on the zero temperature limits of equilibrium states and the study on cohomology equation \cite{Coelho1990} , the Walter observable functions by Bousch \cite{Bousch2001}, a
fullshift with the space of ``super continuous" functions \cite{QuasSiefken2012}, a one-side shift on two symbols with a space of functions
with strong modulus of regularity \cite{BochiZhang2016},  a subshift of finite type with H\"{o}lder continuous functions and zero entropy \cite{Morris2008}.

On the other hand, there are a lot of important results on TPO Conjecture for maps.
Conze and Guivarc'h \cite{ConzeGuivarch1993} improved Coelho's
result by showing the existence of  a continuous function $\varphi$ such that
$\tilde{f} := f +\varphi-\varphi\circ T \leq \sup\int_{\mu\in\mathcal{M}(T)}f d\mu$, where $\mathcal{M}(T)$ is the set of invariant measures, the maximizing measures are those invariant probability measures whose support lies in the set of global maxima of the function $\tilde{f}$. And they also investigated a map $T(x) = 2x (\mbox{mod} 1)$ with observable $f_{\tht}(x) =
\cos(2\pi(x-\tht))$. Later, this model was completely solved based on the work of Hunt and Ott \cite{HuntOtt1996}, Jenkinson via Sturmian measures \cite{Jenkinson1996, Jenkinson2000}, and Bousch by studying the function $\tilde{f}$ \cite{Bousch2000}.
Contreras, Lopes and Thieullen considered a smooth orientation-preserving uniformly expanding map of the circle with the H\"{o}lder functions space \cite{ContrLopesThieu2001}, using techniques inspired by Ma\~{n}\'{e} \cite{Mane1992, Mane1996}, who had established a similar characterization of $\tilde{f}$ in the context of Lagrangian systems. In \cite{Contreras2016}, Contreras verified the TPO Conjecture for the expanding maps with Lipschitz observables. Recently, the TPO Conjecture for the hyperbolic maps with $C^1$ obervables  \cite{HuangLianMa2019a}, and the Aixom A flows with $C^1$ observables \cite{HuangLianMa2019b} were solved.
For a survey of recent development in this problem, please refer to \cite{Jenkinson2019}.

In this paper, we adopt the arguments in \cite{Contreras2016} to verify the TPO Conjecture for an expanding map with a generic (open and dense) $C^r$ ($r\in\mathbb{N}$) differentiable functions. The main idea is the combination of the entropy argument in \cite{Contreras2016} and the smooth perturbation instead of the Lipschitz perturbation used in \cite{Contreras2016}.

\begin{theorem}\label{entarg-3}
Let $X$ be a compact metric space and $T:X\to X$ be an expanding map, then there is an open and dense set $\mathcal{O}\subset C^{r}(X,\mathbb{R})$ ($r\in\mathbb{N}$) such that for all $F\in\mathcal{O}$, there exists a single $F$-maximizing measure and it is supported on a periodic orbit.
\end{theorem}

The rest is organized as follows. In Section \ref{pre-1}, some basic concepts and useful results are introduced. In Section \ref{pre-2}, a useful shadowing property is obtained for differentiable functions. In Section \ref{pre-3}, the entropy argument is used to verify the main result (Theorem \ref{entarg-3}). In Section \ref{pre-4}, a result of Morris is generalized for expanding map with differentiable functions (Theorem \ref{ergzeroentr}), which is used in the entropy argument in Section \ref{pre-3}.

\section{Preliminary}\label{pre-1}

In this section, some basic definitions and useful results are introduced.

\begin{definition} \label{expdef-1}\cite{Contreras2016}
Let $(X,d)$ be a compact metric space. A map $T:X\to X$ is called expanding, if $T$ is Lipschitz continuous and there are constant numbers $\ld\in(0,1)$ and $N_0\in\mathbb{N}$ such that for every point $x\in X$, there are a neighborhood $U_x\subset X$ of $x$  and continuous inverse branches $S_i$, $i=1,...,l_x$, $l_x\leq N_0$, of $T$ with disjoint images $S_i(U_x)$, such that
$T^{-1}(U_x)=\cup^{l_x}_{i=1}S_i(U_x)$, $T\circ S_i=I_{U_x}$ (the identity map restricted to $U_x$), and
\beqq
d(S_i(y),S_i(z))\leq\ld d(y,z)\ \forall y,z\in U_x.
\eeqq
\end{definition}

For $x\in X$ and $r\in\mathbb{R}^{+}$, set
\beqq
B(x,r):=\{w\in X:\ d(x,w)<r\}.
\eeqq

\begin{remark}\label{coverconst-1}
By the compactness of $X$, there is a finite subcover of $\{U_x\}_{x\in X}$ in the above definition. So, there exists a constant $e_0>0$ such that for every $x\in X$, there is some $U_y$ so that the ball $B(x,e_0)\subset U_y$.
\end{remark}

Consider a continuous map $T:X\to X$, the set of invariant measures $\mu\in \mathcal{M}(T)$ with respect to $T$ is given by
\beqq
\mathcal{M}(T)=\{\mu:\ \mu(T^{-1}(B))=\mu(B)\ \mbox{for any Borel subset}\ B\subset X\}.
\eeqq

The following norm is used on the space of Lipschitz functions on $X$, denoted by $\mbox{Lip}(X, \mathbb{R})$,
\beqq
\|f\|_{Lip}=\sup_{x\in X}|f(x)|+\sup_{x\neq y}\frac{|f(x)-f(y)|}{d(x,y)};
\eeqq
the following norm is used on the space of differentiable functions, denoted by $C^r(X, \mathbb{R})$,
\beqq
\|f\|_{C^r}=\sup_{x\in X}|f(x)|+\sum_{1\leq i\leq r}\sup_{x\in X}|f^{(i)}(x)|,
\eeqq
where $f^{(i)}$ is the $i$-th derivative of $f$; if $f:X\to\mathbb{R}$ is continuous, then
\beqq
\|f\|_0=\sup_{x\in X}|f(x)|.
\eeqq

For an observable $f:X\to\mathbb{R}$ with $f\in C^r(X,\mathbb{R})$,  the ergodic optimization is the study of the following problem:
\beqq
\sup_{\mu\in\mathcal{M}(T)}\int_{M}fd\mu.
\eeqq

\begin{remark}
It is evident that $C^r(X,\mathbb{R})\subset \mbox{Lip}(X,\mathbb{R})$. This fact will be used in the following discussions.
\end{remark}

\begin{definition}\cite{Contreras2016}
Given $F\in C^r(X,\mathbb{R})\subset\mbox{Lip}(X,\mathbb{R})$, the Lax operator for $F$ is
$\mathcal{L}_F:\mbox{Lip}(X,\mathbb{R})\to \mbox{Lip}(X,\mathbb{R})$:
\beqq
\mathcal{L}_F(u)(x):=\max_{y\in T^{-1}(x)}\{\al+F(y)+u(y)\},
\eeqq
where
\beqq
\al=\al(F)=-\max_{\mu\mathcal{M}(T)}\int F d\mu.
\eeqq
The set of maximizing measures for an observable function $F$ is
\beqq
\mathcal{M}_{max}(F)=\{\mu\in\mathcal{M}(T):\ \int Fd\mu=-\al(F)\}.
\eeqq
A calibrated sub-action for $F$ is a fixed point of the Lax operator $\mathcal{L}_F$.
\end{definition}

\begin{lemma}\label{supcontin-1}\cite[Lemma 2.1]{Contreras2016}
\begin{itemize}
\item[1.] For $u\in\mbox{Lip}(X,\mathbb{R})$, the Lipschitz constants satisfy
\beq
\mbox{Lip}(\mathcal{L}_F(u))\leq\ld(\mbox{Lip}(F)+\mbox{Lip}(u)).
\eeq
In particular, $\mathcal{L}_F(\mbox{Lip}(X,\mathbb{R}))\subset \mbox{Lip}(X,\mathbb{R})$.
\item[2.] If $\mathcal{L}_F(u)=u$, set
\beq\label{supmea-3}
\overline{F}:=F+u-u\circ T+\al(F),
\eeq
then, we have
\begin{itemize}
\item[(i)] $\al(\overline{F})=-\max_{\mu\in\mathcal{M}(T)}\int\overline{F}d\mu=0$;
\item[(ii)] $\overline{F}\leq0$;
\item[(iii)] $\mathcal{M}(F)=\mathcal{M}(\overline{F})=\{T\ \mbox{invariant measures supported on}\ [\overline{F}=0]\}$.
\end{itemize}
\item[3.] If $u\in\mbox{Lip}(X,\mathbb{R})$ and $\be\in\mathbb{R}$ satisfy $\mathcal{L}_F(u)=u+\be$, then $\be=0$.
\end{itemize}
\end{lemma}

\begin{proposition}\cite[Proposition 2.2]{Contreras2016}
There exists a Lipschitz calibrated sub-action.
\end{proposition}

\begin{remark}
The Lax operator has an invariant subspace
\beq
\bigg\{u\in \mbox{Lip}(X,\mathbb{R}):\ \mbox{Lip}(u)\leq\frac{\ld\,\mbox{Lip}(F)}{1-\ld}\bigg\}.
\eeq
\end{remark}

\begin{definition}\cite{Contreras2016}
For a calibrated sub-action $u$, every point $z\in X$ has a calibrating pre-orbit $\{z_k\}_{k\leq0}$ such that $T(z_{-k})=z_{-k+1}$, $T^i(z_{-i})=z_0=z$, and
\beqq
u(z_{k+1})=u(z_k)+\al+F(z_k)\ \forall k\leq-1,
\eeqq
or
\beqq
\overline{F}(z_k)=0\ \forall k\leq-1.
\eeqq
Hence,
\beqq
u(z_0)=u(z_{-k})+k\al+\sum^{-1}_{i=k}F(z_i)\ \forall k\leq-1.
\eeqq
\end{definition}

\begin{definition}  \cite{Contreras2016}
Given $\de>0$, a sequence $\{x_n\}_{n\in\mathbb{N}}\subset X$ is said to be a $\de$-pseudo-orbit if $d(x_{n+1},T(x_n))\leq\de$ for any $n\in\mathbb{N}$.

Given $\ep>0$, we say that the orbit of a point $y\in X$ $\ep$-shadows a pseudo-orbit $\{x_n\}_{n\in\mathbb{N}}$ if $d(T^n(y),x_n)<\ep$ for any $n\in\mathbb{N}$.
\end{definition}

\begin{lemma}\label{maxval-21}\cite[Lemma 2.3]{Contreras2016}
If there exists a periodic orbit $\mathcal{O}(y)$ such that for any calibrated sub-action, the $\al$-limit of every calibrating pre-orbit is $\mathcal{O}(y)$, then every maximizing measure has support on $\mathcal{O}(y)$.
\end{lemma}

\begin{proposition}\label{supmea-4} \cite[Proposition 2.4]{Contreras2016} (Shadowing Lemma)
For a $\de$-pseudo-orbit $\{x_k\}_{k\in\mathbb{N}}$ with $\de<(1-\ld)e_0$, there is a point $y\in X$ $\ep$-shadowing $\{x_k\}_{k\in\mathbb{N}}$ with $\ep=\tfrac{\de}{1-\ld}$. Moreover, if $\{x_k\}_{k\in\mathbb{N}}$ is a periodic pseudo-orbit, then $y$ is a periodic orbit with the same period.
\end{proposition}

\begin{corollary}\label{maxval-11}
If $T^p(y)=y$ and $\{z_k\}_{k\leq0}$ is a pre-orbit which $(1-\ld)e_0$-shadows the orbit
$\mathcal{O}(y)$ of $y$, that is, for any $k\leq0$, $T(z_k)=z_{k+1}$ and $d(z_k,T^{k\mod p}(y))<(1-\ld)e_0$, then the $\al$-limit of $\{z_k\}_{k\leq0}$ is $\mathcal{O}(y)$.
\end{corollary}

\section{Shadowing property} \label{pre-2}

In this section, a useful shadowing property is derived for differentiable functions.

Let $y\in\mbox{Per}(T)=\cup_{p\in\mathbb{N}}\mbox{Fix}(T^{p})$ be a periodic point for $T$, $P_y$ be the set of differentiable functions $F\in C^r(X,\mathbb{R})$ so that there is a unique $F$-maximizing measure and it is supported on the periodic orbit of $y$, and $\mathcal{U}_y$ be the interior of $P_y$ in $C^r(X,\mathbb{R})$.

\begin{proposition}\label{periappro-1}
Let $F,u\in\mbox{Lip}(X,\mathbb{R})$ with $\mathcal{L}_F(u)=u$ and $\overline{F}=F+\al(F)+u-u\circ T$, where $\al(F)=-\max_{\mu\in\mathcal{M}(T)}\int Fd\mu$.

Suppose that there is $M\in\mathbb{N}^+$ such that for every $Q>1$ and $\de_0>0$, there exist $0<\de<\de_0$ and a $p(\de)$-pseudo-orbit $\{x^{\de}_k\}_{0\leq k\leq p(\de)-1}$ in $[\overline{F}=0]$ with at most $M$ jumps such that $\tfrac{\ga_{\de}}{\de}\geq Q$, where $\ga_{\de}:=\min_{0\leq i<j<p(\de)}d(x^{\de}_i,x^{\de}_j)$. Then $F$ is in the closure of $\bigcup\limits_{y\in Per(T)}\mathcal{U}_y$.
\end{proposition}

\begin{definition}  \cite{Contreras2016}
We say that $n_i$, $i=1,...,l$, $l\leq M$, are the jumps of $\{x^{\de}_k\}_{0\leq k\leq p(\de)-1}$, if $d(T(x_k),x_{k+1})=0$ for $k\in\{0,1,...,p(\de)-1\}\setminus\{n_1,...,n_l\}$.
\end{definition}

Given a positive constant $\ep(\approx\sqrt{\de})$, consider the following constants:
\beq\label{supmea-12}
K:=\max\bigg\{\frac{M\, \mbox{Lip}(\overline{F})}{(1-\ld)^2},\ \frac{\mbox{Lip}(\overline{F})+3}{1-\ld}\bigg\},
\eeq
\beqq
\rho:=\frac{3K\de}{\ep},
\eeqq
\beqq
\ga_{2}:=\ga_{\de}-\frac{2\de}{1-\ld},
\eeqq
\beqq
\ga_{3}:=\frac{\ga_2}{\mbox{Lip}(T)}-\ld \rho,
\eeqq
\beqq
\Ga_1:=\frac{\rho}{12p},
\eeqq
\beqq
\Ga_2:=\frac{K\de}{4p}.
\eeqq
Assume $\de$ are small enough such that the above constants are all positive, $\rho\ll e_0$, $\ga_3\gg \de$, $\Ga_1<1$, $\Ga_2<1$, and set
\beq
-a:=K\de+K\rho+2p\ep\Ga_1+2p\Ga_2-\ep\ga_3<2K\de+K\rho-\ep\ga_3<0
\eeq
and
\beq
-b:=-\ep\rho+\frac{K\de}{p}+2\ep\Ga_1+2\Ga_2<-3K\de+K\de+\frac{K\de}{2}+\frac{K\de}{2}=-K\de<0.
\eeq

Suppose $y$ is a periodic point with period $p$, which $\tfrac{\de}{1-\ld}$ shadows the pseudo-orbit $\{x_k\}$. Set
\beqq
\mathcal{O}(y):=\{T^{i}(y):\ i=0,1,...,p-1\}=\{y_0,y_1,...,y_{p-1}\},
\eeqq
and for any continuous function $G:X\to\mathbb{R}$, denote by
\beqq
\langle G\rangle(y):=\frac{1}{p}\sum^{p-1}_{i=0}G(T^{i}(y))
\eeqq
the average of the function $G$ along the periodic orbit.

\begin{lemma}\label{supmea-9}
Assume that $d(z,y_k)\leq\rho\ll e_0$. Choose $w_1\in T^{-1}(z)$ with $d(w_1,y_{k-1})<\ld\rho$. If $w_2\in T^{-1}(z)\setminus\{w_1\}$, then
\beqq
d(w_2,\mathcal{O}(y))\geq\ga_3=\frac{\ga_2}{\mbox{Lip}(T)}-\ld\rho\gg\de.
\eeqq
\end{lemma}

\begin{proof}
This is a claim verified in the proof of \cite[Proposition 2.6]{Contreras2016}.
\end{proof}

\begin{proof}
In the concept of $\overline{F}$ in \eqref{supmea-3}, the calibrated sub-action $u$ is Lipschitz continuous, but it might not be differentiable. By  2. (iii) of Lemma \ref{supcontin-1}, the maximizing measures for $F$ and $\overline{F}$ are the same.

By Proposition \ref{supmea-4}, we have
\begin{align*}
&\bigg|\sum^{n_i}_{n_{i-1}+1}\overline{F}(y_k)-\sum^{n_i}_{n_{i-1}+1}\overline{F}(x_k)\bigg|
\leq\sum^{n_i}_{n_{i-1}+1}\mbox{Lip}(\overline{F})d(y_k,x_k)\\
\leq&\mbox{Lip}(\overline{F})\sum^{n_i-n_{i-1}}_{l=1}\ld^{k-1}\frac{\de}{1-\ld}\leq\frac{\mbox{Lip}(\overline{F})}{(1-\ld)^2}\de.
\end{align*}
This, together with the assumptions $\overline{F}(x_k)=0$, implies that $\sum^{p-1}_{0}\overline{F}(x_k)=0$, and
\beqq
\sum^{p-1}_{k=0}\overline{F}(y_k)\geq-\frac{M\, \mbox{Lip}(\overline{F})}{(1-\ld)^2}\de\geq-K\de,
\eeqq
or
\beq\label{supmea-5}
\langle\overline{F}\rangle(y)\geq-\frac{K\de}{p}.
\eeq
We make two perturbations to $\overline{F}$.  It follows from \cite{AzagraFerrera2007} that there is a $C^{\infty}$ function $g$ such that
 \beq\label{supmea-6}
 \|g-d(\cdot,\mathcal{O}(y))\|_0<\Gamma_{1}\ \mbox{and}\ \mbox{Lip}(g)<\mbox{Lip}(d(\cdot,\mathcal{O}(y)))+1\leq2,
 \eeq
where the distance function satisfies
\beq
\|d(\cdot,\mathcal{O}(y))\|_0\leq\mbox{diam} X\ \mbox{and}\ \mbox{Lip}(d(\cdot,\mathcal{O}(y)) )\leq1.
\eeq
The first perturbation is $-\ep g(x)$.
The second one is a perturbation by any function $h\in C^{\infty}$ with
\beq\label{supmea-7}
\|h\|_0<\Ga_2\ \mbox{and}\ \mbox{Lip}(h)<1.
\eeq

These perturbations depend on $\mathcal{O}(y)$ and  the period $p$. We will show that the function $G_1=\overline{F}-\ep g+h$ has a unique maximizing measure supported on the periodic orbit $\mathcal{O}(y)$, where such function $G_1$ contains an open ball centered at $\overline{F}-\ep g$. Note that $\overline{F}=F+\al(F)+u-u\circ T$.

Denote by
\beq\label{supmea-11}
G:=\overline{F}-\ep g+h+\be=G_1+\be,
\eeq
where
\beq\label{maxval-1}
\be=-\sup_{\mu\in\mathcal{M}(T)}\int(\overline{F}-\ep g+h)d\mu.
\eeq
It is evident that $G$ and $G_1$ have the same maximizing measures.

By \eqref{supmea-5}, \eqref{supmea-6}, and \eqref{supmea-7}, one has
\begin{align}\label{supmea-8}
\be\leq&-\langle\overline{F}-\ep g+h\rangle(y)=-\langle\overline{F}-\ep d(\cdot,\mathcal{O}(y))+\ep d(\cdot,\mathcal{O}(y))-\ep g +h\rangle(y)\nonumber\\
=&-\langle\overline{F}\rangle(y)+\ep\Ga_1+\|h\|_0\nonumber\\
\leq&\frac{K\de}{p}+\ep\Ga_1+\Ga_2.
\end{align}
Let $v$ be a calibrated sub-action for $G$, that is, $\mathcal{L}_G(v)=v$. Given any $z\in X$, let $\{z_k\}_{k\leq0}$ be a pre-orbit of $z$ calibrating $v$. Denote by $0>t_1>t_2>\cdots$ the times on which $d(z_k,\mathcal{O}(y))>\rho$.

If $t_{n+1}<t_n-1$, there is $s_n\in\mathbb{Z}$ such that the orbit segment $\{z_k\}^{t_n-1}_{k=t_{n+1}+1}$ $\rho$-shadows $\{y_{-i+s_n}\}^1_{i=t_n-t_{n+1}-1}$, then one has
\beqq
d(z_{-i+t_n},y_{-i+s_n})\leq\ld^{i-1}\rho,\ \forall n\in\mathbb{N},\ \forall i=1,...,t_n-t_{n+1}-1.
\eeqq
By Lemma \ref{supmea-9}, for
$t_{n+1}<t_n-1$, we have
\beq\label{supmea-14}
d(z_{t_{n+1}},\mathcal{O}(y))\geq\ga_3.
\eeq
Since both terms in $\overline{F}$ and $-d(\cdot,\mathcal{O}(y))$ are non-positive, it follows from \eqref{supmea-7} and \eqref{supmea-8} that
\begin{align}\label{supmea-10}
G=&\overline{F}-\ep d(\cdot,\mathcal{O}(y))+\ep d(\cdot,\mathcal{O}(y))-\ep g +h+\be\nonumber\\
\leq& \ep\|d(\cdot,\mathcal{O}(y))-g\|_0+\|h\|_0+\be\leq\ep \Ga_1+\|h\|_0+\be\nonumber\\
\leq&\frac{K\de}{p}+2\ep\Ga_1+2\Ga_2.
\end{align}
On a shadowing segment, by \eqref{supmea-12}, \eqref{supmea-6}, \eqref{supmea-7}, \eqref{supmea-11}, we have
\beq\label{supmea-13}
\bigg|\sum^{t_n-1}_{t_{n+1}+1}G(z_k)-\sum^{s_n-1}_{s_n-t_n+t_{n+1}+1}G(y_k)\bigg|\leq
\mbox{Lip}(G)\sum^{+\infty}_{i=0}\ld^i\rho\leq\mbox{Lip}(G)\frac{\rho}{1-\ld}\leq K\rho.
\eeq
Let
\beqq
t_n-t_{n+1}-1=mp+r\ \mbox{with}\ 0\leq r<p,
\eeqq
 and separate the shadowing segment in $m$ loops along the orbit $\mathcal{O}(y)$.

 It follows from the definition of $\be$ and $y$ is a periodic orbit with period $p$ that $\langle G\rangle(y)\leq0$. Hence, from \eqref{supmea-10} and \eqref{supmea-13}, it follows that
\begin{align}\label{supmea-15}
&\sum^{t_n-1}_{t_{n+1}+1}G(z_k)\leq \sum^{s_n-1}_{s_n-t_n+t_{n+1}+1}G(y_k)+\mbox{Lip}(G)\frac{\rho}{1-\ld}\nonumber\\
\leq& mp\langle G\rangle(y)+(p-1)\bigg(\frac{K\de}{p}+2\ep\Ga_1+2\Ga_2\bigg)+
\mbox{Lip}(G)\frac{\rho}{1-\ld}\nonumber\\
\leq &(p-1)\bigg(\frac{K\de}{p}+2\ep\Ga_1+2\Ga_2\bigg)+
K\rho.
\end{align}

Since $d(z_{t_n},\mathcal{O}(y))>\rho$, using \eqref{supmea-8}, \eqref{supmea-6}, and \eqref{supmea-7},  we have
\begin{align}\label{maxvalu-2}
&G(z_{t_n})\leq
\overline{F}(z_{t_n})-\ep d(z_{t_n},\mathcal{O}(y))+\ep d(z_{t_n},\mathcal{O}(y))-\ep g(z_{t_n}) +h(z_{t_n})+\be\nonumber\\
\leq &\overline{F}(z_{t_n})-\ep \rho+\ep| d(z_{t_n},\mathcal{O}(y))- g(z_{t_n})|+ h(z_{t_n})+\be\nonumber\\
\leq&0-\ep\rho+\ep \Ga_1+\Ga_2+\frac{K\de}{p}+\ep\Ga_1+\Ga_2=-\ep\rho+\frac{K\de}{p}+2\ep\Ga_1+2\Ga_2<-b<0.
\end{align}

Similarly, for $t_{n+1}<t_n-1$, by \eqref{supmea-14}, one has
\begin{align}\label{supmea-16}
&G(z_{t_{n+1}})\leq
\overline{F}(z_{t_{n+1}})-\ep d(z_{t_{n+1}},\mathcal{O}(y))+\ep d(z_{t_{n+1}},\mathcal{O}(y))-\ep g(z_{t_{n+1}}) +h(z_{t_{n+1}})+\be\nonumber\\
\leq &\overline{F}(z_{t_{n+1}})-\ep \rho+\ep| d(z_{t_{n+1}},\mathcal{O}(y))- g(z_{t_{n+1}})|+ h(z_{t_{n+1}})+\be\nonumber\\
\leq&0-\ep\ga_3+\ep \Ga_1+\Ga_2+\frac{K\de}{p}+\ep\Ga_1+\Ga_2=-\ep\ga_3+\frac{K\de}{p}+2\ep\Ga_1+2\Ga_2.
\end{align}

For $t_{n+1}<t_n-1$, combining \eqref{supmea-15} and \eqref{supmea-16}, one has
\begin{align}\label{maxval-3}
&\sum^{t_n-1}_{t_{n+1}}G(z_k)\leq(p-1)\bigg(\frac{K\de}{p}+2\ep\Ga_1+2\Ga_2\bigg)+
K\rho-\ep\ga_3+\frac{K\de}{p}+2\ep\Ga_1+2\Ga_2\nonumber\\
=&K\de+K\rho+2p\ep\Ga_1+2p\Ga_2-\ep\ga_3<-a<0.
\end{align}

By \eqref{maxval-1}, $\al(G)=0$. Since $\{z_k\}_{k\leq0}$ is a calibrating pre-orbit for $v$, we have
\beqq
v(z)=v(z_k)+\sum^{-1}_{i=k+1}G(z_i)\ \forall k<0.
\eeqq
Since $v$ is finite, we have
\beqq
\sum^{-1}_{-\infty}G(z_k)\geq-2\|v\|_0>-\infty.
\eeqq
By \eqref{maxvalu-2} and \eqref{maxval-3}, the sequence $t_n$ is finite. Note that $\rho<(1-\ld)e_0$, it follows from Corollary \ref{maxval-11} that any calibrating pre-orbit has $\al$-limit $\mathcal{O}(y)$. This, together with Lemma \ref{maxval-21}, yields that every maximizing measure for $G$ has support on $\mathcal{O}(y)$.
\end{proof}

\section{Entropy argument}\label{pre-3}

In this section,  the entropy argument is used to verify the main result, Theorem \ref{entarg-3}, that is, the set $\mathcal{O}=\cup_{y\in Per(T)}\mathcal{U}_y$ is open and dense.

The difficult part is the proof of the denseness of this set. We will prove this by contradiction.

Suppose that there is a non-empty open set
\beqq
\mathcal{W}\subset C^r(X,\mathbb{R})
\eeqq
which is disjoint from $\cup_{y\in Per(T)}\mathcal{U}_y$. It follows from Theorem
\ref{ergzeroentr} and Remark \ref{ergzeroentropy} that there is $F\in\mathcal{W}$ such that it has an ergodic maximizing measure $\mu$ with zero measure entropy
\beq\label{entarg-2}
h_{\mu}(T)=0.
\eeq
By 2. (iii) of Lemma \ref{supcontin-1}, $\mbox{supp}(\mu)\subset[\overline{F}=0]$ for any calibrating subaction $u$ for $F$, where $\overline{F}$ is specified in \eqref{supmea-3}.  Take $q\in\mbox{supp}(\mu)\subset[\overline{F}]$ satisfying
\beqq
\int fd\mu=\lim_{N\to\infty}\frac{1}{N}\sum^{N-1}_{i=0}f(T^i(q)),
\eeqq
where $f:X\to\mathbb{R}$ is a continuous function, and $q$ is called a generic point for $\mu$.

By the assumption that $F$ is not in the closure of $\cup_{y\in Per(T)}\mathcal{U}_y$ and Proposition \ref{periappro-1} with $M=2$, one has

{\bf Claim} There is $Q>1$ and $\de_0>0$ such that if $0<\de<\de_0$ and $\{x_k\}_{k\geq0}\subset\mathcal{O}(q)$ is a $p(\de)$-periodic $\de$-pseudo-orbit with at most $2$ jumps made with elements of the positive orbit of $q$, then
\beqq
\ga=\min_{1\leq i<j< p}d(x_i,x_j)<\frac{1}{2}Q\de.
\eeqq
Take $N_0$ satisfying that
\beqq
2Q^{-N_0}<\de_0.
\eeqq
Fix a point $w\in\mbox{supp}(\mu)$ satisfying Brin-Katok Theorem \cite{BrinKatok1983}, that is,
\beq\label{entarg-1}
h_{\mu}(T)=-\lim_{L\to\infty}\frac{1}{L}\log\mu(V(w,L,\ep)),
\eeq
where
\beqq
V(w,L,\ep)=\{x\in X:\ d(T^kx,T^kw)<\ep,\ \forall k=0,...,L\}
\eeqq
is the dynamical ball \cite{Bowen1975}. Since $T$ is expanding, we have
\beqq
V(w,L,\ep)=S_1\circ\cdots S_L(B(T^L(w),\ep)),
\eeqq
where $S_k$ is one branch of the inverse of $T$ satisfying that $S_k(T^k(w))=T^{k-1}(w)$. So,
\beq\label{enarg-9}
V(w,L,\ep)\subset B(w,\ld^L\ep).
\eeq

Given $N>N_0$, let $0\leq t^N_1<t^N_2<\cdots$ be all the $\frac{1}{2}Q^{-N}$ returns to $w$, that is,
\beqq
\{t^N_1,t^N_2,...\}=\{n\in\mathbb{N}:\ d(T^n(q),w)\leq\tfrac{1}{2}Q^{-N}\}.
\eeqq

\begin{proposition}\label{entarg-6}\cite[Proposition 3.2]{Contreras2016}
For any $l\geq0$, one has
\beqq
t^{N}_{l+1}-t^N_{l}\geq\sqrt{2}^{N-N_0-1}.
\eeqq
\end{proposition}

Choose $N\gg N_0$ and a continuous function $f_{N}:X\to\mathbb{R}$ satisfying $0\leq f_{N}\leq1$, $f_N|_{B\big(w,\tfrac{1}{2}Q^{-N-1}\big)}\equiv1$, and $\mbox{supp}(f_N)\subset B(w,\tfrac{1}{2}Q^{-N})$. So, by Proposition \ref{entarg-6}, one has
\begin{align}\label{enarg-10}
&\mu(B(w,\tfrac{1}{2}Q^{-N-1}))\nonumber\\
\leq&\int f_Nd\mu=\lim_{L\to+\infty}\frac{1}{L}\sum^{L-1}_{i=0}f_N(T^i(q))\nonumber\\
\leq&\lim_{L\to+\infty}\frac{1}{L}\#\{0\leq i<L:\ d(T^i(q),w)\leq\tfrac{Q^{-N}}{2}\}\nonumber\\
\leq&\lim_{L\to+\infty}\frac{1}{L}\#\{l:\ t^{N}_l\leq L\}\leq\sqrt{2}^{-N+N_0+1}.
\end{align}

Take a sufficiently large $N$ such that
\beqq
\frac{1}{2}Q^{-N-2}\leq\ld^L\ep\leq\frac{1}{2}Q^{-N-1},
\eeqq
so
\beqq
-N\leq L\frac{\log\ld}{\log Q}+\frac{\log(2\ep)}{\log Q}+2.
\eeqq
It follows from \eqref{enarg-9} and \eqref{enarg-10} that
\beqq
\mu(V(w,L,\ep))\leq\mu(B(w,\ld^L\ep))\leq\mu(B(w,\tfrac{1}{2}Q^{-N-1}))\leq\sqrt{2}^{-N+N_0+1},
\eeqq
and
\begin{align*}
&\frac{1}{L}\log\mu(V(w,L,\ep))\leq\frac{1}{L}\log(\sqrt{2})(-N+N_0+1)\\
\leq&\frac{\log\ld}{\log Q}\log(\sqrt{2})+\frac{1}{L}\log(\sqrt{2})\bigg(2+\frac{\log(2\ep)}{\log Q}+N_0+1\bigg).
\end{align*}
By \eqref{entarg-1}, one has
\beqq
h_{\mu}(T)=-\lim_{L\to+\infty}\frac{1}{L}\log\mu(V(w,L,\ep))\geq\frac{\ld^{-1}}{\log Q}\log\sqrt{2}>0,
\eeqq
this is a contradiction with \eqref{entarg-2}.

This completes the proof of the denseness of $\mathcal{O}=\cup_{y\in Per(T)}\mathcal{U}_y$. It is obvious that $\mathcal{O}$ is open.

Therefore, we finish the proof of Theorem \ref{entarg-3}.

\section{Zero entropy} \label{pre-4}

In this section, a result of Morris \cite{Morris2008} is generalized for expanding map with differentiable functions, which is used in the entropy argument in Section \ref{pre-3}.

\begin{theorem}\cite{Morris2008}
Let $X$ be a compact metric space and $T:X\to X$ be an expanding map. There is a residual set $\mathcal{G}\subset\mbox{Lip}(X,\mathbb{R})$ such that if $F\in\mathcal{G}$, then there is a unique $F$-maximizing measure and it has zero metric entropy.
\end{theorem}

Inspired by this result, we show the following result:

\begin{theorem}\label{ergzeroentr}
Let $X$ be a compact metric space and $T:X\to X$ be an expanding map. There is a residual set $\mathcal{G}\subset C^r(X,\mathbb{R})$ ($r\in\mathbb{N}$) such that if $F\in\mathcal{G}$, then there is a unique $F$-maximizing measure and it has zero metric entropy.
\end{theorem}

\begin{remark}\label{ergzeroentropy}
By Lemma \ref{supcontin-1}, the ergodic components of a maximizing measure are also maximizing. Hence, the unique maximizing measure in Theorem \ref{ergzeroentr} is ergodic, further, $T|_{supp(\mu)}$ is uniquely ergodic.
\end{remark}

The arguments below are motivated by the results in \cite{Contreras2016, Morris2008}.

\begin{lemma}\cite[Lemma 4.1]{Contreras2016}
Let $a_1,...,a_n$ be non-negative real numbers and $A=\sum^{n}_{i=1}a_i\geq0$, then
\beqq
\sum^n_{i=1}-a_i\log a_i\leq 1+A\log n,
\eeqq
where $0\log0=0$ is used for convenience.
\end{lemma}

\begin{lemma}\label{enarg-200}\cite[Lemma 4.2]{Contreras2016}
Let $f\in \mbox{Lip}(X,\mathbb{R})$ and suppose that $\mathcal{M}_{max}(f)=\{\mu\}$ for some $\mu\in\mathcal{M}(T)$. Then there is $C>0$ such that for every $\nu\in\mathcal{M}(T)$,
\beqq
-\al(f)-C\int d(x,K)d\nu\leq\int fd\nu,
\eeqq
where $K=\mbox{supp}\,\mu$.
\end{lemma}

For any $\ga\in\mathbb{R}^{+}$, denote by
\beqq
\mathcal{E}_{\ga}:=\{f\in C^r(X,\mathbb{R}):\ h_{\mu}(T)<2\ga\,  h_{top}(T)\ \forall\mu\in\mathcal{M}_{max}(f)\}.
\eeqq

\begin{theorem}\label{entarg-5}\cite{ContrLopesThieu2001, Jenkinson2006}
Let $T:X\to X$ be a continuous map of a compact metric space. Let $E$ be a topological vector space, which is densely and continuously embedded in $C^0(X,\mathbb{R})$. Let
\beqq
\mathcal{U}(E)=\{F\in E:\ \mbox{there is a unique}\ F-\mbox{maximizing measure}\}.
\eeqq
Then $\mathcal{U}(E)$ is a countable intersection of open and dense sets.

Moreover, if $E$ is a Baire space, then $\mathcal{U}(E)$ is dense in $E$.
\end{theorem}

\begin{definition}\cite{Conway1990}
A topological vector space is a vector space together with a topology such that with this respect to this topology such that addition is continuous, and the scalar multiplication is also continuous.
\end{definition}

By Theorem \ref{entarg-5}, the set
\beqq
\mathcal{O}=\{f\in C^r(X,\mathbb{R}):\ \#\mathcal{M}_{max}(f)=1\}
\eeqq
is residual.

So, it suffices to show that $\mathcal{E}_{\ga}$ is open and dense for any $\ga>0$, implying that the set
\beqq
\mathcal{G}=\mathcal{O}\bigcap_{n\in\mathbb{N}}\mathcal{E}_{\tfrac{1}{n}}
\eeqq
satisfies the requirements of Theorem \ref{ergzeroentr}.

\begin{proof}
{\bf Step 1.} We show that $\mathcal{E}_{\ga}$ is open.

Let $f\in C^r(X,\mathbb{R})$, $f_n\in C^r(X,\mathbb{R})\setminus \mathcal{E}_{\ga}$ with $\lim_{n\to\infty}f_n=f$ in $C^r(X,\mathbb{R})$. So, there are $\nu_n\in\mathcal{M}_{max}(f_n)$ with $h(\nu_n)\geq2\ga\ h_{top}(T)$. By the compactness of the space $\mathcal{M}(T)$ in the weak star topology, we can assume that $\lim_{n\to\infty}\nu_n=\nu\in\mathcal{M}(T)$ in the weak star topology. So,
\beqq
\int fd\mu-\|f-f_n\|_{0}\leq\int f_n d\mu\leq \int f_nd\nu_n\leq\int f d\nu_n+\|f-f_n\|_{0},
\eeqq
implying that $\int fd\mu\leq\int fd\nu$ for any $\mu\in\mathcal{M}(T)$, that is, $\nu\in\mathcal{M}_{max}(T)$. It follows from the upper semicontinuity of $m\to h_m(T)$ that $h_{\nu}(T)\geq 2\ga\,h_{top}(T)$ \cite{Walters1982}. Hence, $f\in C^r(X,\mathbb{R})\setminus \mathcal{E}_{\ga}$, yielding that $C^r(X,\mathbb{R})\setminus \mathcal{E}_{\ga}$ is closed. Therefore, $\mathcal{E}_{\ga}$ is open.

{\bf Step 2.} We prove that $\mathcal{E}_{\ga}$ intersects every non-empty open set of $C^r(X,\mathbb{R})$.

Let $\mathcal{U}\subset C^r(X,\mathbb{R})$ be an open and non-empty subset. It follows from Theorem \ref{entarg-5} that there is $f\in\mathcal{U}$ such that $\mathcal{M}_{max}(f)$ contains only one element, denoted by $\mu$.

By the existence of Markov partitions of arbitrarily small diameter for expanding maps \cite{Ruelle2004}, there is a finite collection of sets $S_i\subset X$, a Markov partition, denoted by $\mathbb{P}$, satisfying that
\begin{itemize}
\item $\cup\, S_i=X$;
\item $\mbox{diam}\,\mathbb{P}=\max\{\mbox{diam}\,S_i\}<e_0$;
\item $S_i=\overline{\mbox{int}\, S_i}$;
\item $\mbox{int}\, S_i\cap\mbox{int}\, S_j=\emptyset$ for $i\neq j$;
\item $f(S_i)$ is a union of sets $S_j$.
\end{itemize}
Set
\beqq
\mathbb{P}^{(n)}:=\bigvee^{n-1}_{i=0}T^{-i}(\mathbb{P})=\bigg\{\cap^{n-1}_{i=0}A_i:\ A_i\in T^{-i}(\mathbb{P})\bigg\}.
\eeqq
The diameter of the elements of the partition $\mathbb{P}^{(n)}$ is less than $\ld^{n-1}e_0$, and this partition generates the Borel $\si$-algebra $\mathbb{P}^{\infty}=\si(\cup_n\mathbb{P}^{(n)})=Borel(X)$.

By \cite{Walters1982}, for every invariant measure $\nu\in\mathcal{M}(T)$, one has
\beqq
h_{\nu}(T)=\inf_{k}\frac{1}{k}\sum_{A\in\mathbb{P}^{(k)}}(-\nu(A)\log\nu(A)).
\eeqq

If $\mu$ is a periodic measure, then
\beqq
h_{\mu}(T)=\inf_{k}\frac{1}{k}\sum_{A\in\mathbb{P}^{(k)}}(-\mu(A)\log\mu(A))=0,
\eeqq
so, $f\in\mathcal{E}_{\ga}\cap\mathcal{U}$. Otherwise, if $\mu$ is not a periodic measure, it follows from 2. (iii) of Lemma \ref{supcontin-1} that any measure in $\mbox{supp}(\mu)$ is also a maximizing measure, implying that
\beqq
K=\mbox{supp}(\mu)
\eeqq
does not contain a periodic orbit. By Lemma \ref{enarg-200}, we have
\beq
-\al(f)-C\int d(x,K)d\nu\leq\int fd\nu\ \forall \nu\in\mathcal{M}(T).
\eeq

Note that $f\in\mathcal{U}\subset C^r(X,\mathbb{R})$, for any $g\in C^{\infty}(X,\mathbb{R})$, and $\be\in\mathbb{R}$, it is evident that if $|\be|$ is sufficiently small, then $f+\be g\in\mathcal{U}$.  By using this basic fact, we will construct a sequence of approximating functions $f_n\in\mathcal{U}\cap\mathcal{E}_{\ga}$ for large enough $n$.

{\bf Step 3.} We pick up a sequence of periodic orbits which will be used in the sequel.

Given any $\tht\in(0,1)$, there is a sequence of integers $\{m_n\}_{n\in\mathbb{N}}$ and a sequence of periodic measures $\mu_n\in\mathcal{M}(T)$ satisfying that
\beqq
\int d(x,K)d\mu_n=o(\tht^{m_n})\ \mbox{and}\ \lim_{n\to\infty}\frac{\log n}{m_n}=0.
\eeqq
By \cite[Corollary 3 and Theorem 4]{BressaudQuas2007}, for any given positive integer $k>0$, one has
\beq
\lim_{n\to\infty}n^k\bigg(\inf_{\mu\in\mathcal{M}^n(T)}\int d(x,K)d\mu\bigg)=0.
\eeq
Hence, there exists a sequence of periodic orbits $\mu_n\in\mathcal{M}^n(T)$ such that
\beqq
\lim_{n\to\infty}n^k\int d(x,K)d\mu_n=0.
\eeqq
Set
\beq\label{consest-1}
r_n:=\log_{\tht}\bigg(\int d(x,K)d\mu_n\bigg).
\eeq
It is evident that,  taking $\log_{\tht}$ on both sides of
\beqq
\tht^{r_n}\leq n^k\tht^{r_n}\leq1,
\eeqq
we have
\beqq
-\frac{1}{k}\leq\frac{\log_{\tht}n}{r_n}\leq0.
\eeqq
So, $\frac{\log_{\tht}n}{r_n}\to0$. Define
\beq \label{consest-2}
m_n=\bigg\lfloor\frac{r_n}{2}\bigg\rfloor,
\eeq
 where $\lfloor x\rfloor$ is the largest integer that is less than or equal to $x$. Hence, $\tfrac{\log_{\tht}n}{m_n}\to0$, and
\beq
\int d(x,K)d\mu_n=\tht^{r_n}\leq \tht^{m_n+\tfrac{1}{2}r_n}=o(\tht^{m_n}).
\eeq

{\bf Step 4.} We verify that there is $N_{\ga}>0$ such that when $n\geq N_{\ga}$
\beq\label{supmea-2}
\nu\{x\in X:\ d(x,L_n)\geq\tht^{m_n}\}>\ga
\eeq
for every invariant measure $\nu\in\mathcal{M}(T)$ with $h_{\nu}(T)\geq 2\ga h_{top}(T)$, where
\beq\label{supmea-1}
L_n:=\mbox{supp}(\mu_n),
\eeq
\beq
0<\tht<\min\{e_0,\ \ld,\ e_0\mbox{Lip}(T)^{-1}\},
\eeq
$\ld$ is introduced in Definition \ref{expdef-1}, and $e_0$ is specified in Remark \ref{coverconst-1}.

This is the Claim 4.5 in \cite{Contreras2016}.

{\bf Step 5.}  Define a sequence of functions $\{f_n\}_{n\geq1}\subset C^r(X,\mathbb{R})$.

For $L_n$ specified in \eqref{supmea-1}, the function $d(x,L_n)$ is Lipschitz with Lipschitz constant $1$.  By \cite{AzagraFerrera2007}, there is a $C^{\infty}$ function $\tilde{f}_n(x)$
satisfying that
\beqq
|d(x,L_n)-\tilde{f}_n(x)|<\ga\tht^{m_n+\tfrac{1}{2}r_n}\ \mbox{and}\ \mbox{Lip}(\tilde{f}_n)\leq 1+\frac{1}{n},
\eeqq
where $r_n$ and $m_n$ are introduced in \eqref{consest-1} and \eqref{consest-2}, respectively.

Define
\beqq
f_n(x)=f(x)-\be\tilde{f}_n(x),\ n\geq1,
\eeqq
where $\be$ is sufficiently small positive constant such that $f_n\in\mathcal{U}$ and $\mathcal{U}$ is specified in Step 2, since $f\in C^r(X,\mathbb{R})$, $\tilde{f}_n\in C^{\infty}(X,\mathbb{R})$, and $X$ is compact.

It follows from \eqref{supmea-2} that for sufficiently large $n$,
\beqq
\int d(x,L_n)d\nu\geq\tht^{m_n}\nu(\{x\in X:\ d(x,L_n)\geq\tht^{m_n}\})\geq\ga\tht^{m_n}
\eeqq
for all $\nu\in\mathcal{M}(T)$ with $h(\nu)\geq 2\ga\, h_{top}(T)$.

By Step 3,
\beqq
\int d(x,K)d\mu_n=\tht^{r_n}\leq \tht^{m_n+\tfrac{1}{2}r_n}.
\eeqq

So, we can choose sufficiently large $n$ such that
\beqq
\be\int d(x,L_n)d\nu-2\be\ga\tht^{m_n+\tfrac{1}{2}r_n}>C\int d(x,K)d\mu_n
\eeqq
for every $\nu\in\mathcal{M}(T)$ with $h(\nu)\geq 2\ga\,h_{top}(T)$.

Hence, one has
\begin{align*}
&\int f_nd\nu=\int (f-\be d(x,L_n))d\nu+\int(\be d(x,L_n)-\be\tilde{f}_n(x))d\nu\\
\leq& \int fd\nu-\be \int d(x,L_n)d\nu+\be\int|d(x,L_n)-\tilde{f}_n(x)|d\nu\\
\leq&-\al(f)-\be\int d(x,L_n)d\nu+\be\ga\tht^{m_n+\tfrac{1}{2}r_n}\\
<& -\al(f)-C\int d(x,K)d\mu_n-\be\ga\tht^{m_n+\tfrac{1}{2}r_n}\leq \int fd\mu_n-\be\ga\tht^{m_n+\tfrac{1}{2}r_n}\\
=& \int (f_n+\be\tilde{f}_n)d\mu_n-\be\ga\tht^{m_n+\tfrac{1}{2}r_n}\\
=& \int f_nd\mu_n+\int\be d(x,L_n)d\mu_n+\int(\be \tilde{f}_n-\be d(x,L_n))d\mu_n-\be\ga\tht^{m_n+\tfrac{1}{2}r_n}\\
\leq &  \int f_nd\mu_n+\be\int| \tilde{f}_n- d(x,L_n)|d\mu_n-\be\ga\tht^{m_n+\tfrac{1}{2}r_n}\\
\leq&\int f_nd\mu_n\leq -\al(f_n).
\end{align*}
That is, $\int f_nd\nu<-\al(f_n).$ So, for $\nu\in\mathcal{M}(T)$ with $h_{\nu}(T)\geq 2\ga\, h_{top}(T)$, then $\nu\not\in\mathcal{M}_{max}(f_n)$, implying that $f_n\in\mathcal{E}_{\ga}\cap \mathcal{U}$.

Therefore, $\mathcal{E}_{\ga}$ is open and dense in $C^r(X,\mathbb{R})$. This completes the whole proof.
\end{proof}

\section*{Acknowledgments}
{\footnotesize
We would like to thank Prof. Weixiao Shen, who suggested this problem,  and thank Prof. Yiwei Zhang for useful discussions.

This work was supported by the National Natural Science Foundation of China (No. 11701328) and Young Scholars Program of Shandong University, Weihai (No. 2017WHWLJH09).}

\end{document}